\documentclass[12pt]{amsart}
\usepackage{amssymb}
\usepackage[shortlabels]{enumitem}
\usepackage{xcolor}
\usepackage[margin=1in]{geometry} 
\usepackage{hyperref}
\usepackage{eucal}
\hypersetup{bookmarksdepth=2}
\hypersetup{colorlinks=true}
\hypersetup{linkcolor=blue}
\hypersetup{citecolor=blue}
\hypersetup{urlcolor=blue}

\makeatletter
\@addtoreset{equation}{section}
\makeatother

\numberwithin{equation}{section}
\newtheorem{theorem}[equation]{Theorem}

\newtheorem{proposition}[equation]{Proposition}

\theoremstyle{definition}
\newtheorem{remark}[equation]{Remark}

\newcommand{\bC}{\mathbf{C}}

\newcommand{\bF}{\mathbf{F}}
\newcommand{\cF}{\mathcal{F}}

\newcommand{\bQ}{\mathbf{Q}}

\newcommand{\fS}{\mathfrak{S}}

\newcommand{\bZ}{\mathbf{Z}}

\newcommand{\arxiv}[1]{\href{http://arxiv.org/abs/#1}{{\tiny\tt arXiv:#1}}}

\newcommand{\DOI}[1]{\href{http://doi.org/#1}{\color{purple}{\tiny\tt DOI:#1}}}
\newcommand{\defn}[1]{\emph{#1}}

\let\ul\underline
\renewcommand{\phi}{\varphi}

\DeclareMathOperator{\len}{len}

\DeclareMathOperator{\Aut}{Aut}

\DeclareMathOperator{\Rep}{Rep}

\newcommand{\GL}{\mathbf{GL}}

\title{A remark on automorphisms of tensor spaces}
\date{July 17, 2025}

\author{Alessandro Danelon}
\address{Department of Mathematics, University of Michigan, Ann Arbor, MI, USA}
\email{\href{mailto:adanelon@umich.edu}{adanelon@umich.edu}}
\urladdr{\url{https://public.websites.umich.edu/~adanelon/}}

\author{Andrew Snowden}
\thanks{AS was supported by NSF grant DMS-2301871.}
\address{Department of Mathematics, University of Michigan, Ann Arbor, MI, USA}
\email{\href{mailto:asnowden@umich.edu}{asnowden@umich.edu}}
\urladdr{\url{http://www-personal.umich.edu/~asnowden/}}

\begin{document}

\begin{abstract}
A \emph{tensor space} is a vector space equipped with a finite collection of multi-linear forms. In recent years, a rich theory of infinite dimensional tensor spaces has emerged. In this note, we show that a large class of permutation groups can occur as the automorphism groups of such tensor spaces. Using this, we show that a tensor space can behave somewhat pathologically as a representation of its automorphism group.
\end{abstract}

\maketitle

\section{Introduction}

A \defn{tensor space} is a complex vector equipped with a finite collection of multilinear forms. To be more precise, for a tuple $\ul{d}=(d_1, \ldots, d_r)$ of positive integers, a \defn{$\ul{d}$-space} is a complex vector space $V$ equipped with multi-linear forms $\omega_i \colon V^{d_i} \to \bC$ for each $1 \le i \le r$. We say that a $\ul{d}$-space is \defn{symmetric} if each $\omega_i$ is symmetric. In the last several years, a rich theory of infinite dimensional tensor spaces has emerged, with connections to algebraic geometry, commutative algebra, and model theory; see \cite{BDDE, isocubic, biquad2, isogeny, homoten, homoten2, KaZ}.

The purpose of this note is to realize a wide class of permutation groups as automorphism groups of tensor spaces. This provides some useful examples to inform our intuition on tensor spaces. The following is simplified version of our main theorem:

\begin{theorem} \label{thm:graph}
Let $X$ be a (perhaps infinite) graph with automorphism group $\Gamma$. Then there is a symmetric $(3,2,2)$-space $V$ with basis indexed by (the vertices of) $X$ such that $\Aut(V) \cong \Gamma$, acting on basis vectors in the natural manner.
\end{theorem}

In the general version of the theorem, $X$ is allowed to be a structure over any finite relational language, though in this case more general $\ul{d}$-spaces are required. Some specific examples are given in \S \ref{s:ex}(a,b,c). Using a variant of this result, in \S \ref{s:ex}(d) we obtain:

\begin{theorem} \label{thm:length}
There is a symmetric $(6, 3)$-space $V$ of countable dimension such that $V$ is an irreducible representation of $\Aut(V)$, but $V^{\otimes 2}$ is an infinite length representation.
\end{theorem}

A slightly more intricate construction, given in \S \ref{s:ex}(e), yields:

\begin{theorem}
Given $m \ge 3$ there is a symmetric $(m^2+m,m,m)$-space $V$ of countable dimension such that $V^{\otimes k}$ is a finite length representation of $\Aut(V)$ precisely for $k<m$.
\end{theorem}

We discuss two specific sources of the motivation for these results. In \cite{homoten, homoten2}, Harman and the second author constructed tensor spaces $V$ with very large automorphism groups $G$; these groups can be thought of as new infinite dimensional Lie groups. In these cases, the tensor powers of $V$ are finite length representations of $G$. Thus, if we define $\Rep(G)$ to be the category of representations generated by the $V^{\otimes k}$, then $\Rep(G)$ is a (non-rigid) abelian tensor category in which every object has finite length. This category resembles the category of polynomial functors, and one might hope to employ it in similar ways; e.g., one might develop a geometric theory similar to the $\GL$-varieties of \cite{polygeom}. We are thus interested in finding more examples where a similar picture holds. One might suspect that as soon as $V$ is finite length, all of its tensor powers are; indeed, this was the case in all previously studied examples. Our results show this is not the case.

Second, we believe it should be possible, in principle, to classify tensor spaces with ``large'' automorphism group, provided that large is defined correctly. In recent work \cite{biquad2}, we classified all symmetric $(2,2)$-spaces $V$ such that $V$ is a length two representation of $\Aut(V)$. In all such cases, tensor powers of $V$ remain finite length. The above results suggest to us that $V$ being finite length is not sufficient for $\Aut(V)$ to count as ``large'' in general. A more plausible largeness condition is ``linearly oligomorphic,'' introduced in \cite{homoten}.

\section{Diagonal forms}

Let $V$ be a vector space with basis $\{e_i\}_{i \in I}$, where $I$ is an index set. Fix $d \ge 3$, and define a symmetric $d$-form $f$ on $V$ by 
\begin{displaymath}
f \big( \sum_{i \in I} x_i e_i \big) = \sum_{i \in I} x_i^d.
\end{displaymath}
Here we are implicity using the equivalence between symmetric $d$-linear forms and homogeneous functions of degree $d$. The form $f$ is clearly preserved by the symmetric group $\fS_I$ on the set $I$. It is also preserved by $\mu_d^I$, the group of diagonal matrices in $\GL(V)$ whose diagonal entries are $d$th roots of unity. Thus the wreath product $\mu_d \wr \fS_I$ is contained in $\Aut(f)$. The following simple result plays a key role in our constructions:

\begin{proposition} \label{prop:diag}
We have $\Aut(f)=\mu_d \wr \fS_I$.
\end{proposition}

\begin{proof}
For a vector $v \in V$, we let $\partial_v f$ denote the partial derivative of $f$ in the $v$ direction. If $v=\sum_{i \in I} a_i e_i$ then
\begin{displaymath}
(\partial_v f) \big( \sum_{i \ge 1} x_i e_i \big) = \sum_{i \in I} d a_i x_i^{d-1}.
\end{displaymath}
In particular, we see that $v$ is a scalar multiple of a basis vector if and only if $\partial_v f$ is a power of a linear form. This shows that if $g \in \Aut(f)$ then $ge_i$ is a scalar multiple of some $e_j$. The result now easily follows.
\end{proof}

\section{The main theorem}

Fix a tuple $\ul{d}=(d_1, \ldots, d_r)$ of positive integers. A \emph{$\ul{d}$-structure} is a set $X$ equipped with relations $P_1, \ldots, P_r$, where $P_i$ has arity $d_i$. One can think of $P_i$ as a function $X^{d_i} \to \{0,1\}$, where~0 represents false and~1 true. Fix such a structure $X$, and let $\Gamma=\Aut(X)$. Let $V$ be a vector space with basis $\{e_x\}_{x \in X}$. Define a $d_i$-form $f_i$ on $V$ by
\begin{displaymath}
f_i(e_{x_1}, \ldots, e_{x_{d_i}}) = P_i(x_1, \ldots, x_{d_i}),
\end{displaymath}
and extending linearly. We also define a symmetric $k$-form $g_k$ on $V$ by
\begin{displaymath}
g_k(e_{x_1}, \ldots, e_{x_k}) = \begin{cases}
1 & \text{if $x_1=\cdots=x_k$} \\
0 & \text{otherwise.} \end{cases}
\end{displaymath}
It is clear that the natural action of $\Gamma$ on $V$ preserves the $f_i$ and the $g_k$. In fact:

\begin{theorem} \label{thm:struc}
$\Gamma$ is the full symmetry group of $(f_1, \ldots, f_r, g_2, g_3)$.
\end{theorem}

\begin{proof}
By Proposition~\ref{prop:diag}, the subgroup of $\GL(V)$ preserving $g_k$ is $\mu_k \wr \fS_X$. From this, it follows that the subgroup preserving both $g_2$ and $g_3$ is simply $\fS_X$. Finally, an element of $\fS_X$ preserves the forms $f_1, \ldots, f_r$ if and only if it is an automorphism of the structure $X$. Thus the result follows.
\end{proof}

We now prove a variant of Theorem~\ref{thm:struc}. Fix $m \ge 3$. Define an $md_i$ relation $P'_i$ on $X$ by
\begin{displaymath}
P'_i(x_{1,1}, \ldots, x_{d_i,m}) = \begin{cases}
P_i(x_{1,1}, x_{2,1}, \ldots, x_{d_i,1}) & \text{if $x_{a,b}=x_{a,c}$ for all $a$, $b$, $c$} \\
0 & \text{otherwise} \end{cases}
\end{displaymath}
One should think of the inputs $x_{1,1}, \ldots, x_{d_i,m}$ as consisting of $d_i$ blocks, each of size $m$, and the value of $P'_i$ is only non-zero if each block is constant. Let $X'$ be the $(m \ul{d})$-structure with relations $(P'_1, \ldots, P'_r)$. We note that $\Aut(X')=\Gamma$. Let $f'_i$ be the $md_i$-form on $V$ associated to $P'_i$, and let $g_m$ be as above.

\begin{theorem}
$\mu_m \wr \Gamma$ is the full symmetry group of $(f'_1, \ldots, f'_r, g_m)$.
\end{theorem}

\begin{proof}
By Proposition~\ref{prop:diag}, the subgroup of $\GL(V)$ preserving $g_m$ is $\mu_m \wr \fS_X$. It is clear that an element of this group preserves the $P'_i$ if and only if the permutation is an automorphism of the structure $X$. The result follows.
\end{proof}

\begin{remark}
There is a variant for $m=2$ as well: $\mu_2 \wr \Gamma$ is the full symmetry group of the forms $(f'_1, \ldots, f'_r, g_2, g_4)$.
\end{remark}

\section{Some representation theory}

Let $\Gamma$ be a permutation group on a set $X$, and let $V$ be the vector space with basis $\{e_x\}_{x \in X}$. Fix $m \ge 1$, and let $G=\mu_m \wr \Gamma$ act on $V$ in the usual manner, i.e., $\Gamma$ permutes the basis and the normal subgroup $\mu_m^X$ acts by diagonal matrices. We write $\len_G(-)$ for the length of a $G$-module. We require the following simple result:

\begin{proposition}
We have the following, for $k \ge 1$:
\begin{enumerate}
\item If $\Gamma$ has finitely many orbits on $X^k$ and $k<m$ then $\len_G(V^{\otimes k})<\infty$.
\item If $\Gamma$ has infinitely many orbits on $X^k$ then $\len_G(V^{\otimes k})=\infty$.
\end{enumerate}
\end{proposition}

\begin{proof}
For $\ul{x} \in X^k$, let $e_{\ul{x}}$ be the corresponding basis vector of $V^{\otimes k}$. Let $\{Y_j\}_{j \in J}$ be the orbits of $G$ on $X^k$, and let $W_j$ be the subspace of $V$ spanned by the $e_{\ul{x}}$ with $\ul{x} \in Y_j$. Then each $W_j$ is a subrepresentation of $V$, and $V$ is the direct sum of the $W_j$'s. Thus if there are infinitely many orbits then $V$ has infinite length. This proves (b).

Suppose now that there are finitely many orbits and $k<m$. For each $\ul{x} \in X^k$ there is a homomorphism $\alpha_{\ul{x}} \colon \mu_m^X \to \bC^{\times}$ such that $ge_{\ul{x}} = \alpha(g) e_{\ul{x}}$ for all $g \in \mu_m^X$. Moreover, the assumption that $k<m$ ensures that if $\ul{x}$ and $\ul{y}$ are distinct elements of $X^k$ then $\alpha_{\ul{x}} \ne \alpha_{\ul{y}}$. In other words, $V^{\otimes k}$ is a multiplicity free representation of $\mu_m^X$. It follows that each $W_j$ is an irreducible representation. Indeed, if $y$ is a non-zero vector in $W_j$ then some $e_{\ul{x}}$ with $\ul{x} \in Y_j$ appears with non-zero coefficient in $y$. Since the different basis vectors have different eigenvalues for $\mu_m^X$, it follows that $e_{\ul{x}} \in W_j$. Since $G$ acts transitively on the basis vectors $Y_j$ of $W_j$, we thus see that $y$ generates all of $W_j$. It thus follows that $V^{\otimes k}$ has finite length, which proves (a).
\end{proof}

\section{Examples} \label{s:ex}

(a) Let $X=\bQ$, regarded as a totally ordered set; that is, $X$ is a $(2)$-structure where $P_1(x,y)$ is the relation $x<y$. Our construction produces a $(3,2,2)$-space $V$ with basis indexed by $\bQ$ such that $\Aut(V)$ is the group of order preserving self-bijections of $\bQ$.

(b) Let $\bF$ be a finite field, and fix a primitive root $a$. Let $X=\bigcup_{n \ge 1} \bF^n$. We define a $(3,2)$-structure on $X$ by letting $P_1(x,y,z)$ be the relation $x=y+z$, and $P_2(x,y)$ be the relation $x=ay$. An automorphism of this structure is exactly a linear automorphism of $X$, and so the automorphism group of $X$ is the infinite general linear group $\GL_{\infty}(\bF)$. Our construction thus produces a $(3,3,2,2)$-space $V$ with basis indexed by $X$ having $\Aut(V)=\GL_{\infty}(\bF)$. Note that $V$ is a complex vector space!

(c) Fix $\ul{d}=(d_1, \ldots, d_r)$, and let $\cF$ be a class of finite $\ul{d}$-structures. If $\cF$ is a \defn{Fra\"iss\'e class} then we can form its Fra\"iss\'e limit $X$, which is a countable structure built by glueing together the structures in $\cF$ in a particular manner; see \cite[\S 2.6]{Cameron} or \cite{Macpherson} for details. The structure $X$ is highly symmetrical: its automorphism group $\Gamma$ is \defn{oligomorphic}, meaning that $\Gamma$ has finitely many orbits on $X^k$ for all $k$. Our construction produces a $(d_1, \ldots, d_r, 3, 2)$-space $V$ with basis indexed by $X$ having $\Aut(V)=\Gamma$.

For example, if $\cF$ is the class of finite totally ordered sets then the Fra\"iss\'e limit is the totally ordered set $\bQ$, and we recover example (a) above. Example (b) does not quite fit into this set-up since to obtain $\bF^{\infty}$ as a Fra\"iss\'e limit requires working over an infinite relational language. For another example, take $\cF$ to be the class of all finite graphs. The Fra\"iss\'e limit in this case is the famous Rado graph \cite[\S 2.10]{Cameron}. We thus obtain a symmetric $(3,2,2)$-space $V$ where $\Aut(V)$ is the automorphism group of the Rado graph.

(d) Let $X$ be the graph with vertex set $\bZ$ in which $n$ and $n+1$ are connected for all $n$; this is the Cayley graph for the group $\bZ$ with generating set $\{1\}$. The automorphism group $\Gamma$ of $X$ is simply $\bZ$. Our variant construction, with $m=3$, produces a symmetric $(6,3)$-space $V$ with basis indexed by $\bZ$ having automorphism group $G=\mu_3 \wr \bZ$. Explicitly, if $\{x_i\}_{i \in \bZ}$ are the standard coordinates on $V$ then the forms are given by
\begin{displaymath}
\sum_{i \in \bZ} x_i^3 x_{i+1}^3, \qquad
\sum_{i \in \bZ} x_i^3.
\end{displaymath}
Since $\Gamma$ acts transitively on $X$ but with infinitely many orbits on $X^2$, we see that $V$ is an irreducible representation of $G$, but $V^{\otimes 2}$ has infinite length.

(e) We now generalize the previous example. Given any $m \ge 3$, we produce a symmetric $(m+1,1)$-structure $X$ such that its automorphism group $\Gamma$ has finitely many orbits on $X^k$ for $1 \le k < m$ but infinitely many orbits on $X^m$. Applying our variant construction produces a symmetric $(m^2+m,m,m)$-space $V$ with basis indexed by $X$ having automorphism group $G=\mu_m \wr \Gamma$. We find that $V^{\otimes k}$ is a finite length representation of $G$ for $k<m$, but has infinite length for $k \ge m$.

We now explain how to produce these structures. Fix $m \ge 1$ and a countable infinite set $\Sigma$. Define an $(m,\Sigma)$-hypergraph to be a set equipped with an $m$-uniform hypergraph structure for each element of $\Sigma$. One can think of this as an $m$-uniform hypergraph with edges colored by $\Sigma$, where parallel edges of different colors are allowed. Let $\cF$ be the class of all $(m,\Sigma)$-hypergraphs having finitely many vertices and edges. This is a Fra\"iss\'e class over a countably infinite relational language. Let $Y$ be the Fra\"iss\'e limit, and let $\Delta=\Aut(Y)$. The orbits of $\Delta$ on the set of $k$-element subsets of $Y$ correspond bijectively to isomorphism classes of $k$-element structures in $\cF$. It follows that $\Delta$ has finitely many orbits on $Y^k$ for $k<m$, but infinitely many orbits on $Y^m$.

Now, let $X=Y \coprod \Sigma$. We define an $(m+1,1)$-structure $(P_1, P_2)$ on $X$ as follows. For $a \in \Sigma$, let $Q_a \subset Y^m$ be the edges in the $a$-hypergraph structure. We take
\begin{displaymath}
P_1 = \coprod_{a \in \Sigma} (Q_a \times\{a\}), \qquad P_2=\Sigma.
\end{displaymath}
In other words, $P_1(x_1, \ldots, x_m, x_{m+1})$ is true precisely when $x_{m+1}$ is an element of $\Sigma$ and $(x_1, \ldots, x_m)$ is an edge in the $x_{m+1}$-hypergraph structure on $Y$, while $P_2(x)$ is true precisely when $x \in \Sigma$.

Let $\Gamma$ be the automorphism group of the structure $X$. The group $\Delta$ acts on $X$ by acting trivially on $\Sigma$, and this gives an inclusion $\Delta \subset \Gamma$. Suppose $\sigma$ is a permutation of $\Sigma$, and let $\sigma^*(Y)$ be the structure obtained by permuting the colors on the edges of $Y$. One easily sees that this is also the Fra\"iss\'e limit of the class $\cF$, and thus isomorphic to $X$; choosing an isomorphism $\tau \colon X \to \sigma^*(X)$, we see that $(\tau, \sigma)$ is an automorphism of the structure $Y$. We thus obtain an exact sequence of groups
\begin{displaymath}
1 \to \Delta \to \Gamma \to \fS_{\Sigma} \to 1,
\end{displaymath}
where $\fS_Y$ is the automorphism group of $Y$. From this description, it is clear that $\Gamma$ has finitely many orbits on $X^k$ for $k<m$. On the other hand, it has infinitely many orbits on $X^m$, since $m$ vertices of $Y$ can belong to any finite number of edges of different colors.

\end{document}